\documentclass[12pt]{article}
\newif \ifreview \reviewtrue

\usepackage{currfile,datetime}
\usepackage{amsthm,amsmath,amssymb,bbm,geometry,epsfig,listings,
paralist,enumerate,comment,nicefrac,verbatim,multirow,xcolor,wasysym,tikz}
\usepackage[pagewise]{lineno}[4.41]
\usepackage{mathrsfs}  
\usepackage{marginnote}\reversemarginpar
\usepackage[linktocpage,colorlinks,hidelinks,linkcolor=red,citecolor=green]{hyperref}
\usepackage{soul}

\usepackage[capitalise,compress]{cleveref} 

\ifreview

\newcommand*\patchAmsMathEnvironmentForLineno[1]{%
    \expandafter\let\csname old#1\expandafter\endcsname\csname #1\endcsname
    \expandafter\let\csname oldend#1\expandafter\endcsname\csname end#1\endcsname
    \renewenvironment{#1}%
    {\linenomath\csname old#1\endcsname}%
    {\csname oldend#1\endcsname\endlinenomath}}%
  \newcommand*\patchBothAmsMathEnvironmentsForLineno[1]{%
    \patchAmsMathEnvironmentForLineno{#1}%
    \patchAmsMathEnvironmentForLineno{#1*}}%
  \patchBothAmsMathEnvironmentsForLineno{equation}%
  \patchBothAmsMathEnvironmentsForLineno{align}%
  \patchBothAmsMathEnvironmentsForLineno{gather}%
  \patchBothAmsMathEnvironmentsForLineno{multline}%
\else 

\fi

\crefname{equation}{}{}

\newtheorem{lemma}{Lemma}[section]

\newtheorem{theorem}[lemma]{Theorem}

\crefname{subsection}{Subsection}{Subsections}
\crefname{enumi}{item}{items}
\creflabelformat{enumi}{(#2\textup{#1}#3)}

\newcommand{\1}{\ensuremath{\mathbbm{1}}}
\providecommand{\N}{{\ensuremath{\mathbbm{N}}}}
\providecommand{\Z}{{\ensuremath{\mathbbm{Z}}}}
\providecommand{\R}{{\ensuremath{\mathbbm{R}}}}

\providecommand{\E}{{\ensuremath{\mathbbm{E}}}}

\newcommand{\Id}{\mathrm{Id}}
\newcommand{\xeqref}[1]{}
\newcommand{\tinynote}[1]{}

\renewcommand{\P}{{\ensuremath{\mathbbm{P}}}}
\renewcommand{\gets}{\curvearrowleft}
\newcommand{\F}{{\ensuremath{\mathbbm{F}}}}
\newcommand{\HS}{\mathrm{HS}}
\newcommand{\fctDelay}{\tau}
\renewcommand{\epsilon}{\varepsilon}

\newcommand{\size}[1]{\left\lvert#1\right\rvert}
\title{A path-dependent stochastic Gronwall inequality and strong convergence rate
for stochastic functional differential equations}
\author
{Martin Hutzenthaler$^{1}$ \\
 Tuan Anh Nguyen$^{2}$\bigskip\\
\small{$^1$ Faculty of Mathematics, University of Duisburg-Essen,}\\
\small{Essen, Germany; e-mail: \texttt{martin.hutzenthaler}\textcircled{\texttt{a}}\texttt{uni-due.de}}\\
\small{$^2$ Faculty of Mathematics, University of Duisburg-Essen,}\\
\small{Essen, Germany; e-mail: \texttt{tuan.nguyen}\textcircled{\texttt{a}}\texttt{uni-due.de}}
}

\allowdisplaybreaks
\begin{document}
\maketitle
\begin{abstract}
 We derive a stochastic Gronwall lemma with suprema over the paths in the upper
 bound of the assumed affine-linear growth assumption. This allows applications
  to It\^o processes with coefficients which depend on earlier time points
  such as stochastic delay equations or Euler-type approximations
  of stochastic differential equations.
  We apply our stochastic Gronwall lemma with path-suprema to
  stochastic functional differential equations and prove
  a strong convergence rate for coefficient functions which depend
  on path-suprema.
\end{abstract}

{\makeatletter
\let\@makefnmark\relax
\let\@thefnmark\relax
\@footnotetext{\emph{Key words and phrases:}
stochastic Gronwall lemma, functional stochastic differential equations,
path-dependent stochastic differential equations,
stochastic delay equations
}
\@footnotetext{\emph{AMS 2010 subject classification}: 60E15, 65C30, 34K50} 
\makeatother}

\section{Introduction}
There are numerous applications of the classical (deterministic) Gronwall
lemma.
Scheutzow \cite{scheutzow2013stochastic} derived a powerful
stochastic version of the Gronwall lemma with an $L^p$-estimate with $p\in(0,1)$.
Makasu \cite{makasu2019stochastic} extended this to the case $p=1$.
Hudde et al.\ \cite{hudde2021stochastic} extended this to the case $p\in(1,\infty)$.
For related stochastic Gronwall lemmas see, e.g.,
\cite{renesse2010existence,zhang2018singular,kruse2018discrete,makasu2020extension,xie2020ergodicity,geiss2021sharp,geiss2022concave}.

Recently, Mehri and Scheutzow \cite{mehri2019stochastic} relaxed the affine-linear
growth assumption and allowed running path-suprema in the upper bound.
More precisely, Mehri and Scheutzow
\cite[Theorem 2.2]{mehri2019stochastic} in particularly
prove that if $\alpha\colon[0,\infty)\to[0,\infty)$ is measurable,
if $X,H\colon[0,\infty)\times\Omega\to[0,\infty)$
are adapted stochastic processes on a filtered probability space
$(\Omega,\mathcal{F},\P,(\mathbb{F}_t)_{t\in[0,\infty)})$ with
continuous sample paths and if 
\begin{equation}  \begin{split}
  X_t\leq\int_0^t\alpha_s\sup_{r\in[0,s]}X_r\,ds +M_s+H_s,
\end{split}     \end{equation}
then it holds for all $p\in(0,1)$, $t\in[0,\infty)$ that
\begin{equation}  \begin{split}
  \E\Big[\sup_{s\in[0,t]}|X_s|^p\Big]\leq \tfrac{1}{(1-p)\cdot p^{1+p}}
  \E\Big[\sup_{s\in[0,t]}|H_s|^p\Big]
  \exp\Big(\tfrac{1}{(1-p)^{1/p}\cdot p}\int_0^t\alpha_s\,ds\Big).
\end{split}     \end{equation}
The main goal of this article is to complement this path-dependent stochastic
Gronwall inequality with a result in the case $p\in(1,\infty)$; 
see Theorem \ref{a02} below for the precise formulation.

The second goal of this article is to demonstrate an application 
of our stochastic Gronwall lemma.
Stochastic functional differential equations (SFDEs,
which are also denoted as stochastic delay equations or path-dependent SDEs in the
literature)
appear in a wide range of
applications;
see, e.g., \cite{banos2019stochastic,blath2016new,blath2020population,frank2001stationary,stoica2005stochastic,tian2007stochastic}. 
Some SFDEs can be transformed to classical stochastic differential equations
by the linear chain trick;
see, e.g., \cite{scheutzow2018stochastic}.
In general, however,
solutions are typically not known explicitly
and
the linear chain trick does not work.
We will prove that Euler-type approximations
of SFDEs
converge with strong rate
$0.5-$ if the drift coefficient is one-sided global Lipschitz continuous,
the diffusion coefficient is globally Lipschitz continuous and both coefficients
grow at most linearly with respect to the supremum-norm on path space.
We emphasize that our path-dependent stochastic Gronwall lemma allows
us to consider the global monotonicity condition \eqref{eq:one-sided}
jointly on $\mu$ and $\sigma$ with path-supremum on the right-hand side.
This was not possible before.
The following Theorem \ref{thm:intro} illustrates our main result in this direction.

\begin{theorem}\label{thm:intro}
Let $d\in\N$,
$T,\fctDelay,c \in [0, \infty)$,
$p\in[2,\infty)$,
$\mu\in C( [0,T]\times C([-\fctDelay,T], \R^d), \R^d )$,
$\sigma\in C( [0,T]\times C([-\fctDelay,T],\R^d),  \HS(\R^d,\R^d)) $, 
assume 
for all 
$t\in[0,T]$, $s\in[0,t]$,
$x,y\in C([-\fctDelay,T],\R^d)$
that 
$
 \lVert\mu(t,x)\rVert
 +\lVert\sigma(t,x)\rVert
 \leq 
c\sup_{s\in [-\fctDelay,t]}\left[a+\lVert x(s)\rVert_{H}^2\right]^{\frac{1}{2}}$,
that $\mu(t,\cdot)$, $\sigma(t,\cdot)$ only depend on the interval $[-\fctDelay,t]$,
that
\begin{equation}\begin{split}\label{eq:one-sided}
&
2\big\langle
 x(t)-y(t),
\mu(t,x)-\mu(t,y)
\big\rangle
+p
\left\lVert
\sigma(t,x)-\sigma(t,y)\right\rVert^2
\leq c\sup_{s\in[-\fctDelay,T]}
\lVert x(s)-y(s)\rVert^2,
\end{split}\end{equation}
and that
\begin{equation}\begin{split}\label{eq:temporal}
&\max
\Big\{\!
\left\lVert
\mu(s,x)-
\mu(t,x)\right\rVert^2
,
\left\lVert
\sigma(s,x)-
\sigma(t,x)\right\rVert^2
\Big\}
\leq c\Big[\lvert t-s\rvert+\!\!
\sup_{\substack{u,v\in[0,t]\colon\\ \lvert u-v\rvert\leq \lvert t-s\rvert}}
\left\lVert
x(u)-x(v)\right\rVert^2\Big],
\end{split}
\end{equation}
let $(\Omega,\mathcal{F},\P, (\F_t)_{t\in[0,T]})$ be a filtered probability space which satisfies the usual conditions,
let $W\colon [0,T]\times\Omega\to U$ be an
standard $(\mathbb{F}_s)_{s\in[0,T]}$-Wiener process,
let $\xi=(\xi_t(\omega))_{\omega\in\Omega,t\in[-\fctDelay,0]}\colon \Omega\to C([-\fctDelay,0],H) $ be $\mathbb{F}_0$-measurable,
assume that $\xi$ and $W$ are independent,
 let 
$X 
\colon [-\fctDelay,T]\times\Omega\to \R^d$ be adapted,
have continuous sample paths,
and satisfy
for all $r\in[-\fctDelay,0]$, $t\in (0,T]$ a.s.\ that 
\begin{align}\label{eq:SFDE}
X_{r} 
=
\xi_r\quad\text{and}\quad 
{X}_t = \xi_0
+  \int_{0}^{t}\mu(s,X )\,ds
+  \int_{0}^{t}\sigma(s,X )\,dW_s
,
\end{align}
and for every $n\in\N$ let
$Y^n,\mathcal{Y}^n
\colon [-\fctDelay,T]\times\Omega\to \R^d$ satisfy
assume for all $r\in[-\fctDelay,0]$, 
$k\in [0,n-1]\cap\Z$,
$t\in(\frac{kT}{n},\frac{(k+1)T}{n}]$ that
$\mathcal{Y}_r^n=\xi_r$,
$
\mathcal{Y}^{n}_t=
(k+1-\tfrac{nt}{T})
Y^{n}_{k}
+
(\tfrac{nt}{T}-k)
Y^{n}_{k+1}$
and
\begin{align}Y^n_0=\xi_0,\quad\text{and}\quad
Y_{k+1}^n= Y_k^n+\mu(\tfrac{kT}{n}, \mathcal{Y}^n )\tfrac{T}{n}+
\sigma(\tfrac{kT}{n}, \mathcal{Y}^n )(W_{\frac{(k+1)T}{n}}-W_{\frac{kT}{n}}).
\end{align}
 Then for every $q\in[1,p)$ there exists $C\in\R$ such that for all $n\in\N$ it holds that
\begin{align}\begin{split}
&
\Big(\E\Big[
\sup_{k\in[1,n]\cap\N}
\Big
\lVert
\mathcal{Y}^{n}_k-X_{\frac{kT}{n}}
\Big\rVert^q\Big]\Big)^{\frac{1}{q}}\leq 
Cn^{\frac{1}{2p}-\frac{1}{2}}
.\end{split}
\end{align}
\end{theorem}
Theorem \ref{thm:intro} follows immediately from Theorem \ref{a05}.
Next we discuss the assumptions of
Theorem \ref{thm:intro}.
The coefficients of the SFDE \eqref{eq:SFDE}
are assumed to depend only on the path up to the current time point
to ensure that the solution process is progressively measurable.
The condition \eqref{eq:temporal} is satisfied for example
for the running path-supremum $[0,T]\times C([-\fctDelay,T],\R^d)\ni (t,x)\mapsto
\sup_{s\in[-\fctDelay,t]}x_s$.
Moreover, the coefficients are assumed to satisfy the global
monotonicity condition \eqref{eq:one-sided}.
They do not need to be globally Lipschitz continuous.
However, for convenience we assume that $\mu$ and $\sigma$ grow
at most linearly. Otherwise the Euler-Maruyama approximations
typically diverge in the strong sense; see \cite{hjk11,HJK13}.

We selectively mention results from
An incomplete selection of approximation results
in
the huge literature
on SFDEs is
\cite{KP00,BB00,Mao03,mao2003numerical,BB05,higham2007almost,KS14,Akh19,renesse2010existence,mehri2019stochastic,WM08,lan2018strong,GMY18,lan2018strong}.
A closely related result is Wu and Mao \cite[Theorem 5.1]{WM08} which establishes
$L^2$-rate $0.5-$ if the coefficients functions are globally Lipschitz continuous
with respect to the path-supremum. If the diffusion coefficient
is globally Lipschitz continuous, then our stochastic Gronwall lemma is not
needed and one can apply the Burkholder-Davis-Gundy inequality to
the diffusion part.
Another closely related result is Mehri and Scheutzow
\cite[Theorem~3.2]{mehri2019stochastic}
which proves that Euler approximations converge in probability if
a local version of the monotonicity condition \eqref{eq:one-sided} holds.

\section{A path-dependent stochastic Gronwall inequality}

The following Theorem \ref{a02} is the main result of this article
and establishes a path-dependent stochastic Gronwall inequality.
The function $V$ is typically chosen as
and the squared norm
$V=\big([0,T]\times O\ni (t,x)\mapsto \|x\|_H^2\in[0,\infty)$
and then the Lyapunov-type condition \eqref{a01} becomes a one-sided
linear growth condition.
\begin{theorem}\label{a02}Let
$( H, \left< \cdot , \cdot \right>_H, \left\| \cdot \right\|_H )$
and
$( U, \left< \cdot , \cdot \right>_U, \left\| \cdot \right\|_U )$
be
separable
$\R$-Hilbert spaces,
let $O\subseteq H$ be an open set, 
let 
$T \in (0, \infty)$,
 $p\in[1,\infty)$,
 $ V \in C^{ 1,2 }( [0,T]\times O, [0,\infty) )$,
let
$\alpha,\lambda \colon [0,T]  \to [0,\infty)$ be measurable,
let $(\Omega,\mathcal{F},\P, (\F_t)_{t\in[0,T]})$ be a filtered probability space which satisfies the usual conditions,
let $(W_t)_{t \in [0, T]}$ be an
$\Id_U$-cylindrical $(\mathbb{F}_t)_{t\in[0,T]}$-Wiener process,
let $X\colon[0,T]\times\Omega\to O$,
$\beta,\gamma \colon [0,T] \times \Omega \to [0,\infty)$ be adapted,
let
$a\colon[0,T]\times\Omega\to H$ be measurable,
let
$b\colon[0,T]\times\Omega\to \HS(U,H)$
be 
progressively measurable,
assume that $X$ has continuous sample paths,
assume that for all $t\in[0,T]$ it holds a.s.\ that $\int_0^{T}
(\lVert a_s\rVert_H+\lVert b_s\rVert_{\HS(U,H)}^2)\,ds<\infty$ and
\begin{align}\label{a03}
 X_{t}=X_0+\int_0^t a_s\,ds +\int_0^tb_s\,dW_s,
\end{align}
and
assume that a.s.\ it holds for all $s\in[0,T]$ that
\begin{equation}  \begin{split}\label{a01}
&\left(\frac{\partial}{\partial s}V\right)(s,X_s)
 +
\left(\frac{\partial}{\partial x}V\right)(s,X_s)\,
a_s   +\frac{1}{2}\textup{trace}\Bigl( b_sb_s^{*}\,
    (\textup{Hess}_xV)(s,X_s)
    \Bigr)    \\
&\quad 
+\frac{p-1}{2}\frac{\left\lVert
b_s^*
\left(\nabla_xV\right)(s,X_s)\right\rVert_{U}^2}
    {V(s,X_s)}
    \leq \alpha_s\left[\sup_{r\in[0,s]}V(r,X_r)\right]+
\beta_s\lambda_s+\gamma_s.
\end{split}     \end{equation}
Then for all
$q\in (0,p)$ it holds that
\begin{align}\begin{split}&
\E\!\left[
\left(\sup_{r\in[0,T]}V(r,X_r)\right)^{q}\right]\\
&
\leq 
\E\!\left[\left(\left(V(0,X_0)\right)^{p}+
\int_{0}^{T}(
\beta_s^p\lambda_s+\gamma_s^p)\,ds\right)^{\frac{q}{p}}\right]
\frac{
\exp \left(\frac{\int_{0}^{T}
\bigl( p\alpha_s+(p-1)(\lambda_s+1)\bigr)ds}{\frac{q}{p}(1- \frac{q}{p})^{\frac{p}{q}}}\right)}{\left(\frac{q}{p}\right)^{\frac{q}{p}+1}(1-\frac{q}{p})}
.\end{split}\label{a01b}
\end{align}
 \end{theorem}
\begin{proof}[Proof of \cref{a02}]
W.l.o.g.\ 
we assume that the right-hand side of \eqref{a01b} is finite, otherwise the assertion is trivial.
First,
\eqref{a01}, the fact that
$V,\alpha,\beta, \gamma,\lambda \geq 0$,
 and the fact that
$
\forall\,A,B\in (0,\infty)\colon 
 A^{1-\frac{1}{p}} B^{\frac{1}{p}}\leq (1-\tfrac{1}{p}) A+\tfrac{1}{p}B
$ 
 yield that  a.s.\ it holds for all
$s\in[0,T]$,
$\epsilon\in (0,1)$
 that
\begin{align} 
&p\left(\epsilon+V(s,X_s)\right)^{p-1}\left(\frac{\partial}{\partial s}V\right)(s,X_s)+
p\left(\epsilon+V(s,X_s)\right)^{p-1}\left(\frac{\partial}{\partial x}V\right)(s,X_s)a_s\nonumber
\\&\quad
+p\left(\epsilon+V(s,X_s)\right)^{p-1}\frac{1}{2}\textup{trace}\left(b_s^{*}(\textup{Hess}_xV)(s,X_s)b_s\right)\nonumber
\\&\quad+\frac{1}{2}p(p-1) \left(\epsilon+V(s,X_s)\right)^{p-2}
\left\lVert
b_s^*
\left(\nabla_xV\right)(s,X_s)\right\rVert_{U}^2\nonumber
\\&
=  p\left(\epsilon+V(s,X_s)\right)^{p-1}  \Biggl[
\left(\frac{\partial}{\partial s}V\right)(s,X_s)
     +
     \left(\frac{\partial}{\partial x}V\right)(s,X_s)a_s\nonumber
     \\&\quad
     +\frac{1}{2}\textup{trace}\left(b_sb_s^{*}(\textup{Hess}_xV)(s,X_s)\right)
     +\frac{p-1}{2}\frac{
     \left\lVert
b_s^*
\left(\nabla_xV\right)(s,X_s)\right\rVert_{U}^2
     }{\epsilon+V(s,X_s)}
     \Biggr]\nonumber
\\&
\leq   p\left(\epsilon+\sup_{r\in[0,s]}V(r,X_r)\right)^{p-1}\xeqref{a01}\left[
      \alpha_s \left(\epsilon+\sup_{r\in[0,s]}V(r,X_r)\right)+  
\beta_s\lambda_s+\gamma_s
\right]\nonumber\\
&= 
   p\alpha_s\left(\epsilon+\sup_{r\in[0,s]}V(s,X_s)\right)^{p}+p \left[\left(\epsilon+\sup_{r\in[0,s]}V(r,X_r)\right)^{p}\right]^{1-\frac{1}{p}}\left[
(\beta_s^p)^{\frac{1}{p}}\lambda_s+
(\gamma_s^p)^{\frac{1}{p}}\right]\nonumber
\\
&\leq\tinynote{Young} p\alpha_s\left(\epsilon+\sup_{r\in[0,s]}V(r,X_r)\right)^{p}+p\left[
\left(1-\frac{1}{p}\right)
\left(\epsilon+\sup_{r\in[0,s]}V(r,X_r)\right)^{p}+
\frac{1}{p}\beta_s^p
\right]\lambda_s\nonumber\\
&\quad+
p\left[
\left(1-\frac{1}{p}\right)
\left(\epsilon+\sup_{r\in[0,s]}V(r,X_r)\right)^{p}+
\frac{1}{p}\gamma_s^p
\right]\nonumber
\\
&=\Bigl( p\alpha_s+(p-1)(\lambda_s+1)\Bigr)\left(\epsilon+\sup_{r\in[0,s]}V(r,X_r)\right)^{p}+
\beta_s^p\lambda_s+\gamma_s^p
.
  \label{a07}  \end{align}
This,
It\^o's formula, the fact that $ V \in C^{ 1,2 }( [0,T]\times O, [0,\infty) )$, and  \eqref{a03} show that for all $\epsilon\in (0,1)$, $t\in[0,T]$ it holds a.s.\ that
\begin{equation}  \begin{split}
&(\epsilon+V( t,X_{ t}))^p
=\left(\epsilon+V(0,X_0)\right)^{p}
+\int_0^tp\left(\epsilon+V(s,X_s)\right)^{p-1}\left(\frac{\partial}{\partial x}V\right)(s,X_s)b_s
\,dW_s
\\&\quad
+\int_0^{ t}\Biggl[ p\left(\epsilon+V(s,X_s)\right)^{p-1}\left(\frac{\partial}{\partial s}V\right)(s,X_s)+
p\left(\epsilon+V(s,X_s)\right)^{p-1}\left(\frac{\partial}{\partial x}V\right)(s,X_s)a_s\\&\qquad\qquad\qquad+p\left(\epsilon+V(s,X_s)\right)^{p-1}\frac{1}{2}\textup{trace}\left(b_s^{*}\left(\textup{Hess}_xV\right)(s,X_s)b_s\right) 
\\&\qquad\qquad\qquad+
\frac{1}{2}p(p-1) \left(\epsilon+V(s,X_s)\right)^{p-2}
\left\lVert
b_s^*
\left(\nabla_xV\right)(s,X_s)\right\rVert_{U}^2
\Biggr]
  \,ds\\
&\leq \left(\epsilon+V(0,X_0)\right)^{p}
+\int_0^t p\left(\epsilon+V(s,X_s)\right)^{p-1}\left(\frac{\partial}{\partial x}V\right)(s,X_s)b_s
\,dW_s\\
&\quad +\int_{0}^{t}\xeqref{a07}\left[
\Bigl( p\alpha_s+(p-1)(\lambda_s+1)\Bigr)\left(\epsilon+\sup_{r\in[0,s]}V(r,X_r)\right)^{p}+
\beta_s^p\lambda_s+\gamma_s^p\right]ds.
\end{split}  \label{eq:integrating.factor2}   \end{equation}
This, \cite[Theorem 2.2]{mehri2019stochastic} (applied for every 
$\epsilon\in (0,1)$,
 $q\in (0,p)$ with 
$
 X \gets\bigl((\epsilon+V(  \min \{t,T\},X_{ \min \{t,T\}}))^p \bigr)_{t\in[0,\infty)}$,
$ A\gets\bigl( \int_{0}^{\min \{t,T\}}
\bigl( p\alpha_s+(p-1)(\lambda_s+1)\bigr)\,ds\bigr)_{t\in[0,\infty)} $,
$ M\gets 
\bigl(\int_0^{t\wedge T} p(\epsilon+V(s,X_s))^{p-1}(\frac{\partial}{\partial x}V)(s,X_s)b_s
\,dW_s\bigr)_{t\in[0,\infty)} ,$
$H\gets
\bigl( (\epsilon+V(0,X_0))^{p}+\int_{0}^{t\wedge T}
(\beta_s^p\lambda_s+\gamma_s^p)\,ds\bigr)_{t\in [0,\infty)},$
$  p\gets \frac{q}{p} $ in the notation of \cite[Theorem 2.2]{mehri2019stochastic}),
the measurability and
regularity assumptions of $X$ and $V$, the fact that 
$p\geq 1$,
and nonnegativity of $\alpha,\lambda$, $\epsilon\in (0,1)$ show for all
$q\in (0,p)$ that
\begin{align}\small\begin{split}
&
\E\!\left[
\left(\epsilon+\sup_{r\in[0,T]}V(r,X_r)\right)^{q}\right]\\
&
\leq 
\E\!\left[\left(\left(\epsilon+V(0,X_0)\right)^{p}+
\int_{0}^{T}
\beta_s^p\lambda_s+\gamma_s^p\,ds\right)^{\frac{q}{p}}\right]
\tfrac{
\exp \left(\frac{\int_{0}^{T}
\bigl( p\alpha_s+(p-1)(\lambda_s+1)\bigr)ds}{\frac{q}{p}(1- \frac{q}{p})^{\frac{p}{q}}}\right)}{\left(\frac{q}{p}\right)^{\frac{q}{p}+1}(1-\frac{q}{p})}
.\end{split}
\end{align}
This, the 
dominated convergence theorem, and finiteness of the right-hand side of \eqref{a01b}
complete the proof of \cref{a02}.
\end{proof}

\section{Strong convergence rate for SFDEs}

The following Theorem \ref{a05} is our main result on strong convergence rates
for SFDEs.
We note that if $\sigma$ is globally Lipschitz continuous, then
we may choose $p$ arbitrarily large and then we obtain rate $0.5-$.
We also note that all upper bounds are explicit and thus allow us
to control dependencies, e.g., on the dimension to see which high-dimensional
SFDEs can be approximated without curse of dimensionality.

\begin{theorem}\label{a05}
Let
$( H, \left< \cdot , \cdot \right>_H, \left\| \cdot \right\|_H )$
and
$( U, \left< \cdot , \cdot \right>_U, \left\| \cdot \right\|_U )$
be
separable
$\R$-Hilbert spaces,
let 
$T \in (0, \infty)$,
$\tau\in [0,\infty)$,
$c,a\in [1,\infty)$,
$\epsilon\in (0,1]$,
$\beta \in [0,\infty)$, 
$p\in[2,\infty)$,
$\mu\in C( [0,T]\times C([-\tau,T], H), H )$,
$\sigma\in C( [0,T]\times C([-\tau,T],H),  \HS(U,H)) $, let $n\in\N$,
 $t_0,t_1,\ldots,t_n\in[0,T]$
satisfy for all $t\in [0,T] $
that $
0=t_0<t_1<\ldots<t_n=T$,
assume 
for all 
$t\in[0,T]$,
$\underline{t}\in[0,t]$,
$x_1,x_0\in C([-\fctDelay,T],H)$
that
\begin{align}
\bigl[
\forall\,s\in[-\fctDelay,t]\colon  x_1(s)=x_0(s)\bigr]\Longrightarrow
\bigl[
(\mu(t,x_1)= \mu(t,x_0))
\text{ and }
(\sigma(t,x_1)= \sigma(t,x_0))\bigr],\label{b17}
\end{align}
\begin{align}
 \lVert\mu(t,x_1)\rVert_H\leq 
c\sup_{s\in [-\fctDelay,t]}\left[a+\lVert x_1(s)\rVert_{H}^2\right]^{\frac{1}{2}}
,\quad 
 \left\lVert\sigma(t,x_1)\right\rVert_{\HS(U,H)}^2
\leq c\sup_{s\in [-\fctDelay,t]}\left[a+\lVert x_1(s)\rVert_{H}^2\right],
\label{b01}\end{align}

\begin{equation}\begin{split}
&
2\left\langle
\mu(t,x_1)-\mu(t,x_0), x_1(t)-x_0(t)
\right\rangle_H
+(p-1)(1+\epsilon)
\left\lVert
\sigma(t,x_1)-\sigma(t,x_0)\right\rVert_{\HS(U,H)}^2\\
&\leq c\left[\sup_{s\in[-\fctDelay,T]}
\lVert x_1(s)-x_0(s)\rVert_H^2\right],
\end{split}\label{b10}\end{equation}

\begin{equation}\begin{split}
&\max\left\{
\left\lVert
\mu(t,x_1)-
\mu(s,x_1)\right\rVert_H^2
,\left\lVert
\sigma(t,x_1)-
\sigma(s,x_1)\right\rVert_H^2
\right\}\\
&\leq c\left[\lvert t-s\rvert+
\sup_{u,v\in[0,t]\colon \lvert u-v\rvert\leq \lvert t-s\rvert}
\left\lVert
x_1(u)-x_1(v)\right\rVert_H^2\right]\left[\sup_{s\in[-\fctDelay,t]}\left(a+\lVert x_1(s)\rVert^2\right)^{\beta}\right],\end{split}\label{b11}
\end{equation}
let $(\Omega,\mathcal{F},\P, (\F_t)_{t\in[0,T]})$ be a filtered probability space which satisfies the usual conditions,
let $W=(W_s)_{s \in [0, T]}\colon [0,T]\times\Omega\to U$ be an
$\Id_U$-cylindrical $(\mathbb{F}_t)_{t\in[0,T]}$-Wiener process,
let $\xi=(\xi_t(\omega))_{\omega\in\Omega,t\in[-\fctDelay,0]}\colon \Omega\to C([-\fctDelay,0],H) $ be 
$\F_0$-measurable, assume that $\xi$ and $W$ are independent,
%
 let 
$X^1, \mathcal{X}^1, X^0
\colon [-\fctDelay,T]\times\Omega\to H$ have continuous sample paths,
assume
that $(X^0_s)_{s\in[0,T]}$ is adapted,
assume that for all $r\in[-\fctDelay,0]$, $t\in (0,T]$   it holds a.s.\ that 
\begin{align}
X_{r}^0
=
\xi_r\quad\text{and}\quad 
{X}_t^0= \xi_0
+  \int_{0}^{t}\mu(s,X^0)\,ds
+  \int_{0}^{t}\sigma(s,X^0)\,dW_s
,\label{b09}
\end{align}
 and assume for all $r\in[-\fctDelay,0]$, 
$k\in [0,n-1]\cap\Z$,
$t\in(t_k,t_{k+1}]$ that
\begin{align}\label{b07}
\mathcal{X}^{1}_t=
\frac{t_{k+1}-t}{t_{k+1}-t_k}
{X}^{1}_{t_k}
+
\frac{t-t_k}{t_{k+1}-t_k}
{X}^{1}_{t_{k+1}},
\end{align}
\begin{align}
X_{r}^1=\mathcal{X}_{r}^1
=
\xi_r,
\quad\text{and}\quad
{X}^{1}_t= X_{t_k}^1
+  \mu(t_k,\mathcal{X}^{1})(t-t_k)
+  \sigma(t_k,\mathcal{X}^{1})(W_t-W_{t_k}).\label{b05b}
\end{align}
Then
\begin{enumerate}[i)]
\item \label{b06b}for all 
$i\in \{0,1\}$, $q\in [2,\infty)$ it holds that
\begin{align}
&
\E\!\left[
\left(
\sup_{s\in [-\fctDelay,T]}\left[a+
\left\lVert
X^{i}_{s}\right\rVert_H^2\right]
\right)^{\frac{q}{2}}\right]\leq  7 e^{38 Tq^2c}
\E\!\left[\left(
a+\sup_{r\in [-\tau,0]}\lVert\xi_r\rVert_H^2\right)^{\frac{q}{2}}\right],
\end{align}

\item  \label{b12b}

for all 
 $\underline{u}\in [0,T]$,
$\overline{u}\in[\underline{u},T] $,
 $i\in\{0,1\}$, $q\in [2,\infty)$  it holds that
\begin{align}
\left\lVert
 \sup_{\underline{s},\overline{s}\in \left[\underline{u},\overline{u}\right]}
\left\lVert
X^i_{\underline{s}}-X^i_{\overline{s}}
\right\rVert_H
\right\rVert_{L^{q} (\P;\R) }
\leq  7cq e^{39 Tqc}\left(
\E\!\left[
\left(a+\sup_{r\in [-\tau,0]}\lVert\xi_r\rVert_H^2\right)^{\frac{q}{2}}\right]\right)^{\frac{1}{q}}\lvert \overline{u}-\underline{u}\rvert^{\frac{1}{2}},
\end{align} and
\item \label{k18b}for all  $q\in[1,p)$ it holds that
\begin{align}\begin{split}
&
\E\!\left[\left(
\sup_{s\in[0,T]}
\left
\lVert
X^{1}_s-X^{0}_s
\right\rVert_H^2\right)^{\frac{q}{2}}\right]\leq 
\tfrac{
\exp\! \left(\frac{ T p (c+\epsilon)}{\frac{q}{p}(1- \frac{q}{p})^{\frac{p}{q}}}\right)
}{\left(\frac{q}{p}\right)^{\frac{q}{p}+1}(1-\frac{q}{p})}T^{\frac{q}{p}}
\E\!\left[
\left(a+\sup_{r\in [-\tau,0]}\lVert\xi_r\rVert_H^2\right)^{q\max\{\beta,1\} }\right]\\
&\qquad\cdot 
\Biggl\{  \left[
c+\epsilon+\tfrac{\epsilon p-\epsilon+p}{\epsilon}c\right]202300 c^2 p^2 e^{230 T p c \max\{\beta ^2,1\}}
\size{\delta_1}
\lceil T/\size{\delta_1}\rceil^{\frac{1}{p}}\Biggr\}^{\frac{q}{2}}
.\end{split}
\end{align}
\end{enumerate}
\end{theorem}

\begin{proof}[Proof of \cref{a05}]By the fact that
$\xi $ is $\F_0$-measurable and by conditioning on $\F_0$ it suffices to assume that $\xi$ is deterministic. 
Throughout the rest of this proof let $\mathcal{X}^0\colon [0,T]\times \Omega\to H$ satisfy that
$\mathcal{X}^0= X^0$ and let $\delta_0,\delta_1\colon [0,T]\to\R$ satisfy for all $t\in[0,T]$ that 
\begin{align}\label{b02}
\delta_0(t)=t,\quad \size{\delta_0}=0,\quad 
\delta_1(t)= \1_{\{0\}}(t)+\sum_{i=1}^{n} \1_{(t_{i-1},t_i]}(t), \quad\text{and}\quad\size{\delta_1}=\sup_{i\in [1,n]\cap\Z}\lvert t_i-t_{i-1}\rvert.
\end{align}
First, \eqref{b05b} shows  for all  
$k\in [0,n-1]\cap\Z$,
$t\in(t_k,t_{k+1}]$  that $\delta_1(t)=t_k$ and hence a.s.\ it holds that
\begin{align}\begin{split}
X_{t}^1
&
=X_{t_k}^1
+  \mu(t_k,\mathcal{X}^{1})(t-t_k)
+  \sigma(t_k,\mathcal{X}^{1})(W_t-W_{t_k})\\
&
= X_{t_k}^1
+  \int_{t_k}^{t}\mu({\delta_1}(s),\mathcal{X}^{1})\,ds
+  \int_{t_k}^{t}\sigma({\delta_1}(s),\mathcal{X}^{1})\,dW_s.
\end{split}\end{align}
This, \eqref{b05b},  induction,  \eqref{b09}, and \eqref{b02}
 show  for all  
$i\in\{0,1\}$, $r\in[-\fctDelay,0]$,
$t\in[0,T]$  that a.s.\ it holds that
\begin{align}
X_{r}^i
=
\xi_r
\quad\text{and}\quad
{X}^{i}_t= \xi_0
+  \int_{0}^{t}\mu({\delta}_i(s),\mathcal{X}^{i})\,ds
+  \int_{0}^{t}\sigma({\delta}_i(s),\mathcal{X}^{i})\,dW_s.\label{b05}
\end{align}
This, a standard property of affine linear interpolations,
 \eqref{b07},  and the fact  that $ \mathcal{X}^0=X^0 $
show for all 
$
i\in \{0,1\}$, $t\in[0,T] $ that a.s. it holds that
\begin{align}\label{b21}
\sup_{s\in[-\fctDelay,t]}\left\lVert \mathcal{X}_s^i\right\rVert_H\leq 
\sup_{s\in[-\fctDelay,t]}
\left\lVert {X}_s^i\right\rVert_H.
\end{align}
In addition, note  for all
$\mathfrak{b}\in \HS(U,H)$, $x\in H$ that
 \begin{align}\label{c05}\begin{split}
& \mathrm{trace}(bb^*)=\lVert \mathfrak{b}\rVert_{\HS(U,H)}^2
\quad\text{and}\\
&
\left\lVert \mathfrak{b}^* x\right\rVert_{U}^2
\leq \lVert \mathfrak{b}^*\rVert_{L(H,U)}^2\lVert x\rVert^2_H
\leq \lVert \mathfrak{b}^*\rVert_{\HS(H,U)}^2\lVert x\rVert^2_H
= \lVert \mathfrak{b}\rVert_{\HS(U,H)}^2\lVert x\rVert^2_H.
\end{split}\end{align} 
This, 
the Cauchy--Schwarz inequality,
\eqref{b01}, \eqref{b21},
 and
the fact that
$\forall\,i\in\{0,1\},r\in[-\fctDelay,0]\colon X^i_r=\xi_r$ (see \eqref{b05})
 show that for all $i\in\{0,1\}$, $t\in[0,T]$, $q\in [1,\infty)$ it holds a.s.\ that
\begin{align}\begin{split}
&
\left\langle
\mu({\delta}_i(t),\mathcal{X}^{i}),2 X^i_t\right\rangle_H+
\frac{1}{2}\mathrm{trace}
\left((\sigma\sigma^*)({\delta}_i(t),\mathcal{X}^{i})2\mathrm{Id}_H\right)
+\frac{q-1}{2}
\frac{\left\lVert 2\sigma^*({\delta}_i(t),\mathcal{X}^{i})X_t^i \right\rVert_U^2}{a+ \left\lVert
X_t^i
\right\rVert_H^2}\\
&\leq\tinynote{CS} \left\lVert
\mu({\delta}_i(t),\mathcal{X}^{i})\right\rVert_H 2\left\lVert X^i_t\right\rVert_H
+
\xeqref{c05}
\left\lVert
\sigma({\delta}_i(t),\mathcal{X}^{i})\right\rVert_{\HS(U,H)}^2+
2(q-1)\left\lVert
\sigma({\delta}_i(t),\mathcal{X}^{i})\right\rVert_{\HS(U,H)}^2\\
&= \left\lVert
\mu({\delta}_i(t),\mathcal{X}^{i})\right\rVert_H 2\left\lVert X^i_t\right\rVert_H+
(2 q -1)\left\lVert
\sigma({\delta}_i(t),\mathcal{X}^{i})\right\rVert_{\HS(U,H)}^2\\
&\leq\xeqref{b01} c\left(
\sup_{s\in\left[-\fctDelay,t\right]}\left[a+
\left\lVert
\mathcal{X}^{i}_{s}\right\rVert_H^2\right]^{\frac{1}{2}}\right)
2\left[a+
\left\lVert
{X}^{i}_{t}\right\rVert_H^2\right]^{\frac{1}{2}}
+(2q-1)\xeqref{b01}c
\sup_{s\in\left[-\fctDelay,t\right]}\left[a+
\left\lVert
\mathcal{X}^{i}_{s}\right\rVert_H^2\right]
\\
&\leq\xeqref{b21} 3qc
\sup_{s\in\left[-\fctDelay,t\right]}\left[a+
\left\lVert
X^{i}_{s}\right\rVert_H^2\right]
\leq\xeqref{b05} 3qc
\sup_{s\in\left[0,t\right]}\left[a+
\left\lVert
X^{i}_{s}\right\rVert_H^2\right]+ 3qc 
\sup_{r\in [-\tau,0]}\lVert\xi_r\rVert_H^2
.\end{split}
\end{align}
This,  \cref{a02}
(applied for every $q\in[1,\infty)$, $i\in\{0,1\}$ with  
$p\gets q$,
$O\gets H$,
$V\gets ([0,T]\times H\ni (t,x)\mapsto 
a+
\lVert x\rVert^2_H \in[0,\infty))$,
$ \alpha\gets ([0,T]\ni t\mapsto 3qc\in [0,\infty)) $,
$ \lambda\gets ([0,T]\ni t\mapsto 3qc\in [0,\infty))$, 
$X\gets (X^{i}_t)_{t\in[0,T]}$, 
$ a\gets (\mu({\delta}_i(s),\mathcal{X}^{i}))_{s\in[0,T]} $,
$ b\gets \left(\sigma({\delta}_i(s),\mathcal{X}^{i})\right)_{s\in[0,T]}$,
$ \beta\gets ( [0,T]\times\Omega\ni (t,\omega)\mapsto \sup_{r\in [-\tau,0]}\lVert\xi_r\rVert_H^2\in [0,\infty) )$,
$\gamma\gets ([0,T]\times\Omega\ni (t,\omega)\mapsto 0\in [0,\infty))$,
$q\gets 0.5q$
in the notation of \cref{a02}), the fact that
$\xi$ is deterministic,
the fact that
$q\cdot 3qc +(q-1)(3qc+1)=3q^2c+3q^2c-3qc+q-1\leq 6q^2c $,
the fact that
$ \frac{0.5q}{q}(1-\frac{0.5q}{q})^{\frac{q}{0.5q}}= 0.5^{2.5}$, 
the fact that
$(\frac{0.5q}{q})^{\frac{0.5q}{q}+1} (1-\frac{0.5q}{q})=0.5^{2.5}$,
the fact that
$\frac{1}{0.5^{2.5}}=2^{2.5}\leq 6$, and the fact that
$\sqrt{1+3Tqc}\leq \sqrt{e^{3Tq^2c}} = e^{1.5Tq^2c}$
 show for all   $q\in[1,\infty)$, $i\in\{0,1\}$ that
\begin{align}\begin{split}
&
\E\!\left[
\left(
\sup_{s\in [0,T]}\left[a+
\left\lVert
X^{i}_{s}\right\rVert_H^2\right]
\right)^{\frac{q}{2}}\right]\\
&\leq 
\tfrac{\exp \left(\frac{\int_{0}^{T} (q\cdot 3qc+(q-1)(3qc+1))\,ds}{\frac{0.5q}{q}(1-\frac{0.5q}{q})^{\frac{q}{0.5q}}}\right)}{(\frac{0.5q}{q})^{\frac{0.5q}{q}+1} (1-\frac{0.5q}{q})}
\E\!\left[
\left[
\left(
a+
\left\lVert
X^{i}_{0}\right\rVert_H^2\right)^{q}
+3T qc \left(\sup_{r\in [-\tau,0]}\lVert\xi_r\rVert_H^2\right)^{q}
\right]^{\frac{1}{2}}\right]\\
&\leq 6e^{6\cdot 6Tq^2c}
\sqrt{1+3Tqc}
\left(a+\sup_{r\in [-\tau,0]}\lVert\xi_r\rVert_H^2\right)^{\frac{q}{2}}
\leq 
6e^{38 Tq^2c}
\left(a+\sup_{r\in [-\tau,0]}\lVert\xi_r\rVert_H^2\right)^{\frac{q}{2}}.
\end{split}\label{b06}\end{align}
This shows 
  for all   $q\in[1,\infty)$, $i\in\{0,1\}$ that
\begin{align}
&
\E\!\left[
\left(
\sup_{s\in [-\fctDelay,T]}\left[a+
\left\lVert
X^{i}_{s}\right\rVert_H^2\right]
\right)^{\frac{q}{2}}\right]\leq 7 e^{38 Tq^2c}
\left(a+\sup_{r\in [-\tau,0]}\lVert\xi_r\rVert_H^2\right)^{\frac{q}{2}}.\label{k14}
\end{align}
This and the fact that $\xi$ is deterministic show \eqref{b06b}.

Next, \eqref{b05},  the triangle inequality, 
 the Burkholder-Davis-Gundy inequality (see, e.g., \cite[Lemma~7.7]{dz92}),
\eqref{b01},  \eqref{b21}, the fact that $c\geq 1$, 
the fact that
$\forall\, q\in [2,\infty)\colon 1+\sqrt{0.5q(q-1)}\leq \sqrt{2(1+0.5q(q-1))}= \sqrt{q^2-q+2}\leq q$,
 the fact that
$\forall\,s,t\in [0,T]\colon \lvert t-s\rvert\leq \sqrt{T\lvert t-s\rvert}\leq e^{0.5T}\sqrt{\lvert t-s\rvert}$,  
and
\eqref{k14} 
  show  for all 
$\underline{u}\in [0,T]$,
$\overline{u}\in[\underline{u},T] $,
 $i\in\{0,1\}$, $q\in [2,\infty)$  that
\begin{align}\begin{split}
&
\left\lVert
 \sup_{\underline{s},\overline{s}\in \left[\underline{u},\overline{u}\right]}
\left\lVert
X^i_{\underline{s}}-X^i_{\overline{s}}
\right\rVert_H
\right\rVert_{L^{q} (\P;\R) }\\
&
=\xeqref{b05}
\left\lVert
 \sup_{\underline{s},\overline{s}\in \left[\underline{u},\overline{u}\right]\colon \underline{s}\leq \overline{s}}
\left\lVert\int_{\underline{s}}^{\overline{s}}\mu({\delta}_i(r),\mathcal{X}^{i})\,dr+\int_{\underline{u}}^{s}\sigma({\delta}(r),\mathcal{X}^{i})\,dW_r\right\rVert_H\right\rVert_{L^{q} (\P;\R)}\\
&\leq
\int_{\underline{u}}^{\overline{u}}\left\lVert\left\lVert\mu({\delta}_i(r),\mathcal{X}^{i})
\right\rVert_H\right\rVert_{L^{q} (\P;\R)}\,dr
+\sqrt{\tfrac{q(q-1)}{2}}
\left[
\int_{\underline{u}}^{\overline{u}}\left\lVert\left\lVert \sigma({\delta}_i(r),\mathcal{X}^{i})\right\rVert_{\HS(U,H)}\right\rVert_{L^{q} (\P;\R)}^2\,dr\right]^{\frac{1}{2}}
\\
&\leq\xeqref{b01}\xeqref{b21} \int_{\underline{u}}^{\overline{u}}c
\left\lVert
\sup_{s\in[-\fctDelay,T]} \left[a+ \left\lVert X^i_s\right\rVert_H^2\right]^{\frac{1}{2}}\right\rVert_{L^{q} (\P;\R)}dr\\
&\qquad\qquad\qquad\qquad
+
\sqrt{\tfrac{q(q-1)}{2}}
\left[
\int_{\underline{u}}^{\overline{u}}c\left\lVert
\sup_{s\in[-\fctDelay,T]} \left[a+ \left\lVert X^i_s\right\rVert_H^2\right]^{\frac{1}{2}}
\right\rVert_{L^{q} (\P;\R)}^2\,dr\right]^{\frac{1}{2}}
\\
&\leq c\left(1+\sqrt{\tfrac{q(q-1)}{2}}\right)\max \left\{\lvert \overline{u}-\underline{u}\rvert,\lvert \overline{u}-\underline{u}\rvert^{\frac{1}{2}}\right\}\left(
\E\! \left[
\left(
\sup_{s\in[-\fctDelay,T]} 
\left[a+ \left\lVert X^i_s\right\rVert_H^2\right]\right)^{\frac{q}{2}}\right]\right)^{\frac{1}{q}}\\
&
\leq cq e^{0.5T}\lvert \overline{u}-\underline{u}\rvert^{\frac{1}{2}}
\xeqref{k14}
 7 e^{38 Tqc}\left(a+\sup_{r\in [-\tau,0]}\lVert\xi_r\rVert_H^2\right)^{\frac{1}{2}}\\
&
\leq 
 7cq e^{39 Tqc}
\left(a+\sup_{r\in [-\tau,0]}\lVert\xi_r\rVert_H^2\right)^{\frac{1}{2}}\lvert \overline{u}-\underline{u}\rvert^{\frac{1}{2}}
.
\label{b12}\end{split}
\end{align}
This and the fact that $\xi$ is deterministic show \eqref{b12b}. 

For the next step
for every 
$i\in\{0,1\}$,
$t\in[0,T]$ let $\mathcal{X}^{i,t}\colon[0,T]\times \Omega\to H $  have continuous sample paths and satisfy for all $s\in[0,T]$ that a.s.\ it holds that
\begin{align}\begin{split}&
\mathcal{X}^{1,t}_s= \1_{[-\fctDelay,\delta_1(t)]}(s)
\mathcal{X}^{1}_s+
\1_{(\delta_1(t),t)} (s)\left[
\frac{t-s}{t-\delta_1(t)}
\mathcal{X}^1_{\delta_1(t)}
+\frac{s-\delta_1(t)}{t-\delta_1(t)}
{X}^1_{t}
\right]+\1_{[t,T]}(s){X}^1_{t},
\\&\mathcal{X}^{0,t}=X^0.\label{b19}
\end{split}\end{align}
Then \eqref{b21} yields 
for all $t\in [0,T]$, $i\in\{0,1\}$
that a.s. it holds that
\begin{align}\label{b19b}
\sup_{s \in [-\fctDelay,t]} \lVert\mathcal{X}^{i,t}_s\rVert_H
\leq 
\sup_{s \in [-\fctDelay,t]} \lVert {X}^{i}_s\rVert_H .
\end{align}
Furthermore, \eqref{b19} and \eqref{b07} show that for all $k,\ell\in [0,n-1]\cap\Z $,
$\underline{s},\overline{s},t\in[0,T]$ with
$\overline{s}\in [t_{k},t_{k+1}]$,
$\underline{s}\in [t_{\ell},t_{\ell+1}]$, 
$\underline{s}\leq \overline{s}\leq t$, $\max\{t_{\ell},t_k\}< t$  it holds a.s.\ that
\begin{align}
\mathcal{X}^{1,t}_{\overline{s}}=\xeqref{b19}
 \xeqref{b07}
\frac{{\overline{s}}-t_{k}}{\min \{t,t_{k+1}\}-t_{k}}
{X}^{1}_{\min \{t,t_{k+1}\}}
+
\frac{\min \{t,t_{k+1}\}-{\overline{s}}}{\min \{t,t_{k+1}\}-t_{k}}
{X}^{1}_{t_{k}}
\end{align} and
\begin{align}
\mathcal{X}^{1,t}_{\underline{s}}
=\xeqref{b19}\xeqref{b07} \frac{\underline{s}-t_{\ell}}{\min\{t,t_{\ell+1}\}-t_{\ell}}
{X}^{1}_{\min\{t,t_{\ell+1}\}}
+ \frac{\min\{t,t_{\ell+1}\}-\underline{s}}{\min\{t,t_{\ell+1}\}-t_{\ell}} {X}^{1}_{t_{\ell}}.
\end{align} 
This and the triangle inequality show that
  for all $k,\ell\in [0,n-1]\cap\Z $,
$\overline{s}\in [t_{k},t_{k+1}]$,
$\underline{s}\in [t_{\ell},t_{\ell+1}]$, 
$\underline{u},\overline{u},t\in [0,T]$
 with $\underline{u}\leq \underline{s}\leq \overline{s}\leq \overline{u}\leq t$  it holds a.s.\ that
\begin{align}\begin{split}
&\left\lVert
\mathcal{X}^{1,t}_{\overline{s}}-
\mathcal{X}^{1,t}_{\underline{s}}\right\rVert_H
= \Biggl\lVert
\frac{\overline{s}-t_k}{\min\{t,t_{k+1}\}-t_k}
(
{X}^{1}_{\min\{t,t_{k+1}\}}-
{X}^{1}_{\min \{t,t_{\ell+1}\}})\\
&\qquad
+
\frac{\min\{t,t_{k+1}\}-\overline{s}}{\min\{t,t_{k+1}\}-t_k}(
{X}^{1}_{t_k}
-{X}^{1}_{\min \{t,t_{\ell+1}\}})
+\frac{\min \{t,t_{\ell+1}\}-\underline{s}}{\min \{t,t_{\ell+1}\}-t_\ell}
({X}^{1}_{\min \{t,t_{\ell+1}\}}-
{X}^{1}_{t_\ell})\Biggr\rVert_H
\\
&\leq 2\left[
\sup_{r_1,r_2\in \left[\delta_1(
\underline{u})
,
\min\{
\overline{u}+\size{\delta_1},t\}\right]}
\left\lVert X_{r_1}^1-X^1_{r_2}\right\rVert_H\right].\\
\end{split}\end{align}
This, \eqref{b12}, and the triangle inequality 
 show  that for all 
$t, \underline{u},\overline{u}\in[0,T] $,
 $q\in [2,\infty)$ with 
$\underline{u}\leq \overline{u}\leq \delta_1(t)$ it holds that
\begin{align}\begin{split}
&
\left\lVert
 \sup_{\underline{s},\overline{s}\in \left[\underline{u},\overline{u}\right]}
\left\lVert
\mathcal{X}^{1,t}_{\overline{s}}-\mathcal{X}^{1,t}_{\underline{s}}
\right\rVert_H
\right\rVert_{L^{q} (\P;\R) }
\leq 14 cq e^{39 Tqc}
\left(a+\sup_{r\in [-\tau,0]}\lVert\xi_r\rVert_H^2\right)^{\frac{1}{2}}
\bigl[\lvert \overline{u}-\underline{u}\rvert+2\size{\delta_1}\bigr]^{\frac{1}{2}}
.
\end{split}\label{b08}
\end{align}
This, \eqref{b12}, \eqref{b19},
and the fact that
$\forall\, A,B\in [0,\infty)\colon \sqrt{A+B}\leq \sqrt{A}+\sqrt{B}$
 show that
for all 
$t, \underline{u},\overline{u}\in[0,T] $,
 $q\in [2,\infty)$, $i\in\{0,1\}$ with 
$\underline{u}\leq \overline{u}\leq t $ it holds that
\begin{align}
&\left\lVert
 \sup_{\underline{s},\overline{s}\in \left[\underline{u},\overline{u}\right]}
\left\lVert
\mathcal{X}^{i,t}_{\overline{s}}-\mathcal{X}^{i,t}_{\underline{s}}
\right\rVert_H
\right\rVert_{L^{q} (\P;\R) }\leq  14 cq e^{39 Tqc}
\left(a+\sup_{r\in [-\tau,0]}\lVert\xi_r\rVert_H^2\right)^{\frac{1}{2}}
\left[
\lvert
\overline{u}-\underline{u}\rvert^{\frac{1}{2}}+2\size{\delta_i}^{\frac{1}{2}}\right].\label{b12c}
\end{align}
This, the triangle inequality, and the fact that 
$\forall\, A,B\in [0,\infty),q\in [1,\infty)\colon (A+B)^q\leq 2^{q-1}(A^q+B^q)$ show for all $i\in\{0,1\}$ that
\begin{align}\begin{split}
&\left(\E\!\left[
\sup_{u,v\in[-\fctDelay,t]\colon \lvert u-v\rvert\leq \size{\delta_1}}
\left\lVert
\mathcal{X}^{i,t}_u
-\mathcal{X}^{i,t}_v
\right\rVert_H^q\right]\right)^{\frac{1}{q}}\\
&
\leq \left(2^{q-1}
\E\!\left[
\sum_{k=0}^{\lfloor T/\size{\delta_1}\rfloor}
\sup_{u,v\in \left[k\size{\delta_1},\min\{(k+1)\size{\delta_1},T\}\right]}
\left\lVert
\mathcal{X}^{i,t}_u
-\mathcal{X}^{i,t}_v
\right\rVert_H^q
\right]\right)^{\frac{1}{q}}\\
&\leq 2 \lceil T/\size{\delta_1}\rceil^{\frac{1}{q}}
14 cq e^{39 Tqc}
\left(a+\sup_{r\in [-\tau,0]}\lVert\xi_r\rVert_H^2\right)^{\frac{1}{2}}
3
\size{\delta_1}^{\frac{1}{2}}\\
&=
84 cq e^{39 Tqc}
\left(a+\sup_{r\in [-\tau,0]}\lVert\xi_r\rVert_H^2\right)^{\frac{1}{2}}
\size{\delta_1}^{\frac{1}{2}}
\lceil T/\size{\delta_1}\rceil^{\frac{1}{q}}
.
\end{split}\label{b13}\end{align}
Furthermore, \eqref{b07}, \eqref{b19},
 and the triangle inequality show for all 
$k\in\Z\cap[0,n-1]$,
$s\in(t_k,t_{k+1}]$, $t\in [s, T]$ that a.s.\ it holds that
\begin{align}
\mathcal{X}^{1,t}_s= \xeqref{b07}\xeqref{b19}
\frac{\min \{t,t_{k+1}\}-s}{\min \{t,t_{k+1}\}-t_k} X^1_{t_k}
+\frac{s-t_k}{\min \{t,t_{k+1}\}-t_k} X^1_{\min \{t,t_{k+1}\}}
\end{align} and hence
\begin{align}\begin{split}
&
\left\lVert\mathcal{X}^{1,t}_s-\mathcal{X}^{0,t}_s\right\rVert_H
 \\
&
=\xeqref{b07}\xeqref{b19}\Biggl\lVert
\frac{\min \{t,t_{k+1}\}-s}{\min \{t,t_{k+1}\}-t_k} (X^1_{t_k}-X_{t_k}^0)
+\frac{s-t_k}{\min \{t,t_{k+1}\}-t_k} (X^1_{\min \{t,t_{k+1}\}} -{X}_{\min \{t,t_{k+1}\}}^0) \\
&\qquad\qquad\qquad
+\frac{\min \{t,t_{k+1}\}-s}{\min \{t,t_{k+1}\}-t_k} (X_{t_k}^0-X_{s}^0)
+\frac{s-t_k}{\min \{t,t_{k+1}\}-t_k} ({X}_{\min \{t,t_{k+1}\}}^0-X_{s}^0)\Biggr\rVert_H \\
&\leq  \left[
\sup_{r\in [-\fctDelay,t]}\left\lVert {X}^1_r-{X}^0_r\right\rVert_H\right]
+\left[\sup_{\underline{s},\overline{s}\in [0,t]\colon \lvert \overline{s}-\underline{s}\rvert\leq \size{\delta_1}}
\left\lVert {X}^0_{\overline{s}}-{X}^0_{\underline{s}}\right\rVert_H\right].
\end{split}\label{b20}
\end{align}
For the next step let $\Gamma \colon [0,T]\times \Omega\to \R$ satisfy that for all
$t\in [0,T]$ it holds a.s.\ that
\begin{align}\begin{split}
\Gamma_{t}&=(c+\epsilon)\left[\sup_{\underline{s},\overline{s}\in [0,t]\colon \lvert \overline{s}-\underline{s}\rvert\leq \size{\delta_1}}
\left\lVert {X}^0_{\overline{s}}-{X}^0_{\underline{s}}\right\rVert_H^2\right]
\\
&\quad +\tfrac{\epsilon p-\epsilon+p}{\epsilon}
c\left[\size{\delta_1}+
\sup_{u,v\in[-\fctDelay,t]\colon \lvert u-v\rvert\leq \size{\delta_1}}
\left\lVert
\mathcal{X}^{1,t}_u
-\mathcal{X}^{1,t}_v
\right\rVert_H^2\right]\left[\sup_{s\in[-\fctDelay,t]}\left(a+\lVert {X}^{1}_s\rVert^2\right)^{\beta}\right].
\end{split}\label{b15}\end{align}
Then H\"older's inequality, the triangle inequality, \eqref{b13}, 
and
 \eqref{k14}
show for all
$t\in[0,T]$ that
\begin{align}\begin{split}
&
\left\lVert\Gamma_{t}\right\rVert_{L^{\frac{p}{2}}(\P;\R) }\leq (c+\epsilon)\left\lVert\sup_{\underline{s},\overline{s}\in [0,t]\colon \lvert \overline{s}-\underline{s}\rvert\leq \size{\delta_1}}
\left\lVert {X}^0_{\overline{s}}-{X}^0_{\underline{s}}\right\rVert_H^2\right\rVert_{L^{\frac{p}{2}}(\P;\R) }
\\
& +\tfrac{\epsilon p-\epsilon+p}{\epsilon}
c\left[\size{\delta_1}+\left\lVert
\sup_{u,v\in[-\fctDelay,t]\colon \lvert u-v\rvert\leq \size{\delta_1}}
\left\lVert
\mathcal{X}^{1,t}_u
-\mathcal{X}^{1,t}_v
\right\rVert_H^2\right\rVert_{L^{p}(\P;\R)}\right]
\left\lVert\sup_{s\in[-\fctDelay,t]}\left(a+\lVert {X}^{1}_s\rVert^2\right)^{\beta}\right\rVert_{L^{p}(\P;\R) } \\
&\leq (c+\epsilon)\left[84 c p e^{39 T pc}
\left(a+\sup_{r\in [-\tau,0]}\lVert\xi_r\rVert_H^2\right)^{\frac{1}{2}}
\size{\delta_1}^{\frac{1}{2}}
\lceil T/\size{\delta_1}\rceil^{\frac{1}{p}}\right]^{2}\\
&\quad +
\tfrac{\epsilon p-\epsilon+p}{\epsilon}c
\left(\size{\delta_1}+
\left[84 c \cdot  2p\cdot  e^{39 T \cdot 2p\cdot c}
\left(a+\sup_{r\in [-\tau,0]}\lVert\xi_r\rVert_H^2\right)^{\frac{1}{2}}
\size{\delta_1}^{\frac{1}{2}}
\lceil T/\size{\delta_1}\rceil^{\frac{1}{2p}}\right]^2
\right)\\
&\qquad\qquad\cdot 7 e^{38 T(2\beta p)^2c/p}
\left(a+\sup_{r\in [-\tau,0]}\lVert\xi_r\rVert_H^2\right)^{\beta}\\
&\leq 
\left[
c+\epsilon+\tfrac{\epsilon p-\epsilon+p}{\epsilon}c\right]
\left[85 c \cdot  2p\cdot  e^{39 T \cdot 2p\cdot c}
\left(a+\sup_{r\in [-\tau,0]}\lVert\xi_r\rVert_H^2\right)^{\frac{1}{2}}
\size{\delta_1}^{\frac{1}{2}}
\lceil T/\size{\delta_1}\rceil^{\frac{1}{2p}}\right]^2\\
&\qquad\qquad\cdot 7 e^{38 T(2\beta p)^2c/p}
\left(a+\sup_{r\in [-\tau,0]}\lVert\xi_r\rVert_H^2\right)^{\beta}\\
&\leq \left[
c+\epsilon+\tfrac{\epsilon p-\epsilon+p}{\epsilon}c\right]202300 c^2 p^2 e^{230 T p c \max\{\beta ^2,1\}}
\left(a+\sup_{r\in [-\tau,0]}\lVert\xi_r\rVert_H^2\right)^{\max\{\beta,1\}}
\size{\delta_1}
\lceil T/\size{\delta_1}\rceil^{\frac{1}{p}}.
\end{split}\label{b16}\end{align}
In addition, \eqref{c05} 
and the fact that $p\in[2,\infty)$
show for all $\mathfrak{a},x\in H$, $\mathfrak{b}\in \HS(U,H)$ that
\begin{align}
2
\left\langle
\mathfrak{a}, x
\right\rangle_H+
\tfrac{1}{2}\mathrm{trace}
\left(
\mathfrak{b} \mathfrak{b}^* 2\mathrm{Id}_H
\right)
+\tfrac{0.5p-1}{2}
\tfrac{\left\lVert 2 \mathfrak{b}^*x \right\rVert_U^2}{\left\lVert
x
\right\rVert_H^2}\leq 2
\left\langle
\mathfrak{a}, x
\right\rangle_H+
(p-1)
\lVert \mathfrak{b}\rVert_{\HS(U,H)}^2.
\label{k19}
\end{align} 
This, \eqref{b17}, \eqref{b19},
the fact that
$\forall\,A,B\in H\colon 2\langle A,B\rangle_H\leq \frac{1}{\epsilon}\lVert A\rVert_H^2+\epsilon\lVert B\rVert^2$, and
the fact that
$\forall\, A,B\in H\colon \lVert A+B\rVert^2_H\leq 
(1+\epsilon)\lVert A\rVert^2_H+
(1+\frac{1}{\epsilon})\lVert B\rVert^2_H
$
 show for all
$t\in[0,T]$ that a.s.\ it holds that
\begin{align}\begin{split}
&\left[2
\left\langle
\mathfrak{a}, x
\right\rangle_H+
\frac{1}{2}\mathrm{trace}
\left(
\mathfrak{b}\mathfrak{b}^* 2\mathrm{Id}_H
\right)
+\frac{0.5p-1}{2}
\frac{\left\lVert 2\mathfrak{b}^*x \right\rVert_U^2}{\left\lVert
x
\right\rVert_H^2}\right]\Bigr|_{
\substack{\mathfrak{a}=\mu({\delta_1}(t),\mathcal{X}^{1})-\mu(t,{X}^{0}),\\\mathfrak{b}=
\sigma({{\delta_1}}(t),\mathcal{X}^{1})
-\sigma(t,{X}^{0}),\, x=  X_t^{1}-X_t^{0}}
}  \\
&
\leq \xeqref{k19}
2
\left\langle
\mu({\delta_1}(t),\mathcal{X}^{1})-\mu(t,{X}^{0}),  X_t^{1}-X_t^{0}
\right\rangle_H+
(p-1)\left\lVert 
\sigma(\delta_1(t),\mathcal{X}^1)
-\sigma(t,\mathcal{X}^0)\right\rVert_{\HS(U,H)}^2  \\
&=\xeqref{b17}\xeqref{b19}
2
\left\langle
\mu({\delta_1}(t),\mathcal{X}^{1,t})-\mu(t,\mathcal{X}^{0,t}),  \mathcal{X}_t^{1,t}-\mathcal{X}_t^{0,t}
\right\rangle_H\\
&\quad +
(p-1)\left\lVert 
\sigma(\delta_1(t),\mathcal{X}^{1,t})
-\sigma(t,\mathcal{X}^{0,t})\right\rVert_{\HS(U,H)}^2  \\
&=2
\left\langle
\mu(t,\mathcal{X}^{1,t})-\mu(t,\mathcal{X}^{0,t}),  \mathcal{X}_t^{1,t}-\mathcal{X}_t^{0,t}
\right\rangle_H+2
\left\langle
\mu({\delta_1}(t),\mathcal{X}^{1,t})-
\mu(t,\mathcal{X}^{1,t}), \mathcal{X}_t^{1,t}-\mathcal{X}_t^{0,t}
\right\rangle_H \\
&\quad +
(p-1)
\left\lVert 
\sigma(\delta_1(t),\mathcal{X}^{1,t})
-\sigma(t,\mathcal{X}^{0,t})\right\rVert_{\HS(U,H)}^2 
\\
&\leq 2\left\langle
\mu(t,\mathcal{X}^{1,t})-\mu(t,\mathcal{X}^{0,t}),  \mathcal{X}^{1,t}_t-\mathcal{X}_t^{0,t}
\right\rangle_H \\
&\quad +
\tfrac{1}{\epsilon}
\left\lVert
\mu({\delta_1}(t),\mathcal{X}^{1,t})-
\mu(t,\mathcal{X}^{1,t})
\right\rVert_H^2 +\epsilon
\left\lVert  \mathcal{X}_t^{1,t}-\mathcal{X}_t^{0,t}\right\rVert_H^2 \\
&\quad +(p-1)
(1+\epsilon)
\left\lVert
\sigma(t,\mathcal{X}^{1,t})
-\sigma(t,\mathcal{X}^{0,t})\right\rVert_H^2\\
&\quad  
+(p-1)(1+\tfrac{1}{\epsilon})
\left\lVert
\sigma(\delta_1(t),\mathcal{X}^{1,t})-
\sigma(t,\mathcal{X}^{1,t})
\right\rVert_{\HS(U,H)}^2 \\
\end{split}
\end{align}
This, \eqref{b10}, \eqref{b11}, the fact that
$ \frac{1}{\epsilon}+(p-1)(1+\frac{1}{\epsilon})= \frac{\epsilon p-\epsilon+p}{\epsilon}$,  \eqref{b20},  \eqref{b19b}, and \eqref{b15} show for all
$t\in[0,T]$ that a.s.\ it holds that
\begin{align}\begin{split}
&\left[2
\left\langle
\mathfrak{a}, x
\right\rangle_H+
\frac{1}{2}\mathrm{trace}
\left(
\mathfrak{b}\mathfrak{b}^* 2\mathrm{Id}_H
\right)
+\frac{0.5p-1}{2}
\frac{\left\lVert 2\mathfrak{b}^*x \right\rVert_U^2}{\left\lVert
x
\right\rVert_H^2}\right]\Bigr|_{
\substack{\mathfrak{a}=\mu({\delta_1}(t),\mathcal{X}^{1})-\mu(t,{X}^{0}),\\\mathfrak{b}=
\sigma({{\delta_1}}(t),\mathcal{X}^{1})
-\sigma(t,{X}^{0}),\, x=  X_t^{1}-X_t^{0}}
}  \\
&\leq \xeqref{b10} \xeqref{b11}(c+\epsilon)\sup_{s\in[-\fctDelay,t]}
\left\lVert
\mathcal{X}^{1,t}_s-\mathcal{X}_s^{0,t}\right\rVert_H 
\\
&\quad +\tfrac{\epsilon p-\epsilon+p}{\epsilon}
c\left[\size{\delta_1}+
\sup_{u,v\in[-\fctDelay,t]\colon \lvert u-v\rvert\leq \size{\delta_1}}
\left\lVert
\mathcal{X}^{1,t}_u
-\mathcal{X}^{1,t}_v
\right\rVert_H^2\right]\left[\sup_{s\in[-\fctDelay,t]}\left(a+\lVert \mathcal{X}^{1,t}_s\rVert^2\right)^{\beta}\right]\\
&\leq  \xeqref{b20}\xeqref{b19b}(c+\epsilon)
\left[
\sup_{r\in [-\fctDelay,t]}\left\lVert {X}^1_r-{X}^0_r\right\rVert_H^2\right]
+(c+\epsilon)\left[\sup_{\underline{s},\overline{s}\in [0,t]\colon \lvert \overline{s}-\underline{s}\rvert\leq \size{\delta_1}}
\left\lVert {X}^0_{\overline{s}}-{X}^0_{\underline{s}}\right\rVert_H^2\right]
\\
&\quad +\tfrac{\epsilon p-\epsilon+p}{\epsilon}
c\left[\size{\delta_1}+
\sup_{u,v\in[-\fctDelay,t]\colon \lvert u-v\rvert\leq \size{\delta_1}}
\left\lVert
\mathcal{X}^{1,t}_u
-\mathcal{X}^{1,t}_v
\right\rVert_H^2\right]\left[\sup_{s\in[-\fctDelay,t]}\left(a+\lVert {X}^{1}_s\rVert^2\right)^{\beta}\right]\\
&= \xeqref{b15}(c+\epsilon)
\left[
\sup_{r\in [-\fctDelay,t]}\left\lVert {X}^1_r-{X}^0_r\right\rVert_H\right]+\Gamma_{t}.
\end{split}
\end{align}
This,  \eqref{b05}, \eqref{b09}, \cref{a02} (applied for every  $q\in [1,p)$ with 
$p\gets 0.5p$,
$O \gets H$,
$V\gets ( [0,T]\times H\ni (t,x)\mapsto \lVert x\rVert_H^2\in[0,\infty) ) $,
$ \alpha \gets ([0,T]\ni t\mapsto c+\epsilon\in [0,\infty))$,
$ \lambda \gets ([0,T]\ni t\mapsto 0\in [0,\infty))$,
$
 X\gets X^{1}-X^{0} $,
$a\gets \bigl(\mu({\delta_1}(s),\mathcal{X}^{1})-\mu(s,X^0)\bigr)_{s\in [0,T]}$,
$ b\gets \bigl(\sigma({\delta_1}(s),\mathcal{X}^{1})
-
\sigma(s,X^{0})\bigr)_{s\in[0,T]} $,
$ \beta\gets\bigl( [0,T]\times \Omega\ni (t,x)\mapsto 0\in [0,\infty)\bigr)$,
$\gamma\gets \Gamma$,
$ q\gets 0.5q $
in the notation of \cref{a02}), the fact that $c\geq 1$, Jensen's inequality,    Tonelli's theorem, and \eqref{b16}
show for all $q\in [1,p)$ that
\begin{align}\begin{split}
&
\E\!\left[\left(
\sup_{s\in[0,T]}
\left
\lVert
X^{1}_s-X^{0}_s
\right\rVert_H^2\right)^{\frac{q}{2}}\right]\leq
\frac{
\exp\! \left(\frac{\int_{0}^{T}
\bigl( 0.5p\cdot (c+\epsilon) +(0.5p-1)\bigr)ds}{\frac{q}{p}(1- \frac{q}{p})^{\frac{p}{q}}}\right)}{\left(\frac{q}{p}\right)^{\frac{q}{p}+1}(1-\frac{q}{p})}
\E\!\left[\left(\int_{0}^{T}\Gamma_s^{0.5p}\,ds\right)^{\frac{0.5q}{0.5p}}\right]\\
&\leq  
\frac{
\exp\! \left(\frac{ T p (c+\epsilon)}{\frac{q}{p}(1- \frac{q}{p})^{\frac{p}{q}}}\right)}{\left(\frac{q}{p}\right)^{\frac{q}{p}+1}(1-\frac{q}{p})}\left(
\int_{0}^{T}\E\!\left[\Gamma_s^{0.5p}\right]ds\right)^{\frac{0.5q}{0.5p}}
\leq \tfrac{
\exp\! \left(\frac{ T p (c+\epsilon)}{\frac{q}{p}(1- \frac{q}{p})^{\frac{p}{q}}}\right)}{\left(\frac{q}{p}\right)^{\frac{q}{p}+1}(1-\frac{q}{p})}
T^{\frac{q}{p}}\sup_{s\in [0,T]}
\left\lVert
\Gamma_s
\right\rVert_{L^{\frac{p}{2}}(\P;\R)}^{\frac{q}{2}}\\
&\leq 
\tfrac{
\exp\! \left(\frac{ T p (c+\epsilon)}{\frac{q}{p}(1- \frac{q}{p})^{\frac{p}{q}}}\right)
}{\left(\frac{q}{p}\right)^{\frac{q}{p}+1}(1-\frac{q}{p})}T^{\frac{q}{p}}
\xeqref{b16}\Biggl\{  \left[
c+\epsilon+\tfrac{\epsilon p-\epsilon+p}{\epsilon}c\right]202300 c^2 p^2 e^{230 T p c \max\{\beta ^2,1\}}\\
&\qquad\qquad\qquad\qquad\qquad\qquad\qquad
\left(a+\sup_{r\in [-\tau,0]}\lVert\xi_r\rVert_H^2\right)^{\max\{\beta,1\}}
\size{\delta_1}
\lceil T/\size{\delta_1}\rceil^{\frac{1}{p}}\Biggr\}^{\frac{q}{2}}
\label{k17}
.\end{split}
\end{align}
The proof of \cref{a05} is thus completed.
\end{proof}

%

{
\small
\bibliographystyle{acm}
\bibliography{../Bib/bibfile}

\begin{thebibliography}{10}

\bibitem{Akh19}
{\sc Akhtari, B.}
\newblock Numerical solution of stochastic state-dependent delay differential
  equations: convergence and stability.
\newblock {\em Advances in Difference Equations 396\/} (2019).

\bibitem{BB05}
{\sc Baker, C.~T., and Buckwar, E.}
\newblock Exponential stability in $p$-th mean of solutions, and of convergent
  {Euler-type} solutions, of stochastic delay differential equations.
\newblock {\em Journal of Computational and Applied Mathematics 184\/} (2005),
  404--427.

\bibitem{banos2019stochastic}
{\sc Ba{\~n}os, D.~R., Cordoni, F., Di~Nunno, G., Di~Persio, L., and R{\o}se,
  E.~E.}
\newblock Stochastic systems with memory and jumps.
\newblock {\em Journal of Differential Equations 266}, 9 (2019), 5772--5820.

\bibitem{blath2016new}
{\sc Blath, J., Casanova, A.~G., Kurt, N., and Wilke-Berenguer, M.}
\newblock A new coalescent for seed-bank models.
\newblock {\em The Annals of Applied Probability 26}, 2 (2016), 857--891.

\bibitem{blath2020population}
{\sc Blath, J., and Kurt, N.}
\newblock Population genetic models of dormancy.
\newblock {\em arXiv preprint arXiv:2012.00810\/} (2020).

\bibitem{BB00}
{\sc Burrage, K., and Burrage, P.~M.}
\newblock Order conditions of stochastic {R}unge-{K}utta methods by
  {$B$}-series.
\newblock {\em SIAM J. Numer. Anal. 38}, 5 (2000), 1626--1646 (electronic).

\bibitem{dz92}
{\sc {Da Prato}, G., and Zabczyk, J.}
\newblock {\em Stochastic equations in infinite dimensions}, vol.~44 of {\em
  Encyclopedia of Mathematics and its Applications}.
\newblock Cambridge University Press, Cambridge, 1992.

\bibitem{frank2001stationary}
{\sc Frank, T., and Beek, P.}
\newblock Stationary solutions of linear stochastic delay differential
  equations: Applications to biological systems.
\newblock {\em Physical Review E 64}, 2 (2001), 021917.

\bibitem{geiss2021sharp}
{\sc Geiss, S.}
\newblock {Sharp nonlinear generalizations of stochastic Gronwall
  inequalities}.
\newblock {\em arXiv preprint arXiv:2112.05047\/} (2021).

\bibitem{geiss2022concave}
{\sc Geiss, S.}
\newblock {Concave and other generalizations of stochastic Gronwall
  inequalities}.
\newblock {\em arXiv preprint arXiv:2204.06042\/} (2022).

\bibitem{GMY18}
{\sc Guo, Q., Mao, X., and Yue, R.}
\newblock {The truncated Euler--Maruyama method for stochastic differential
  delay equations}.
\newblock {\em Numerical Algorithms 78\/} (2018), 599--624.

\bibitem{higham2007almost}
{\sc Higham, D.~J., Mao, X., and Yuan, C.}
\newblock Almost sure and moment exponential stability in the numerical
  simulation of stochastic differential equations.
\newblock {\em SIAM Journal on Numerical Analysis 45}, 2 (2007).

\bibitem{hudde2021stochastic}
{\sc Hudde, A., Hutzenthaler, M., and Mazzonetto, S.}
\newblock A stochastic {G}ronwall inequality and applications to moments,
  strong completeness, strong local {L}ipschitz continuity, and perturbations.
\newblock {\em Ann. Inst. Henri Poincar\'{e} Probab. Stat. 57}, 2 (2021),
  603--626.

\bibitem{hjk11}
{\sc Hutzenthaler, M., Jentzen, A., and Kloeden, P.~E.}
\newblock Strong and weak divergence in finite time of {E}uler's method for
  stochastic differential equations with non-globally {L}ipschitz continuous
  coefficients.
\newblock {\em Proc. R. Soc. Lond. Ser. A Math. Phys. Eng. Sci. 467\/} (2011),
  1563--1576.

\bibitem{HJK13}
{\sc Hutzenthaler, M., Jentzen, A., and Kloeden, P.~E.}
\newblock Divergence of the multilevel {M}onte {C}arlo {E}uler method for
  nonlinear stochastic differential equations.
\newblock {\em Annals of Applied Probability 23}, 5 (2013), 1913--1966.

\bibitem{kruse2018discrete}
{\sc Kruse, R., and Scheutzow, M.}
\newblock {A discrete stochastic Gronwall lemma}.
\newblock {\em Mathematics and Computers in Simulation 143\/} (2018), 149--157.

\bibitem{KP00}
{\sc K\"uchler, U., and Platen, E.}
\newblock Strong discrete time approximation of stochastic differential
  equations with time delay.
\newblock {\em Mathematics and Computers in Simulation 54}, 1--3 (2000),
  189--205.

\bibitem{KS14}
{\sc Kumar, C., and Sabanis, S.}
\newblock Strong convergence of {Euler} approximations of stochastic
  differential equations with delay under local {Lipschitz} condition.
\newblock {\em Stochastic Analysis and Applications 32}, 2 (2014), 207--228.

\bibitem{lan2018strong}
{\sc Lan, G., and Wang, Q.}
\newblock {Strong convergence rates of modified truncated EM methods for
  neutral stochastic differential delay equations}.
\newblock {\em arXiv:1807.08983\/} (2018).

\bibitem{makasu2019stochastic}
{\sc Makasu, C.}
\newblock A stochastic gronwall lemma revisited.
\newblock {\em Infinite Dimensional Analysis, Quantum Probability and Related
  Topics 22}, 01 (2019), 1950007.

\bibitem{makasu2020extension}
{\sc Makasu, C.}
\newblock {Extension of a stochastic Gronwall lemma}.
\newblock {\em Bulletin Polish Acad. Sci. Math. 68\/} (2020), 97--104.

\bibitem{Mao03}
{\sc Mao, X.}
\newblock Numerical solutions of stochastic functional differential equations.
\newblock {\em LMS Journal of Computation and Mathematics 6\/} (2003),
  141--161.

\bibitem{mao2003numerical}
{\sc Mao, X., and Sabanis, S.}
\newblock {Numerical solutions of stochastic differential delay equations under
  local Lipschitz condition}.
\newblock {\em Journal of Computational and Applied Mathematics 151}, 1 (2003),
  215--227.

\bibitem{mehri2019stochastic}
{\sc Mehri, S., and Scheutzow, M.}
\newblock {A stochastic Gronwall lemma and well-posedness of path-dependent
  SDEs driven by martingale noise}.
\newblock {\em arXiv preprint arXiv:1908.10646\/} (2019).

\bibitem{scheutzow2013stochastic}
{\sc Scheutzow, M.}
\newblock A stochastic {G}ronwall lemma.
\newblock {\em arXiv:1304.5424\/} (2013), 4 pages.

\bibitem{scheutzow2018stochastic}
{\sc Scheutzow, M.}
\newblock Stochastic delay equations.
\newblock {\em Lecture Notes, CIMPA School\/} (2018).

\bibitem{stoica2005stochastic}
{\sc Stoica, G.}
\newblock A stochastic delay financial model.
\newblock {\em Proceedings of the American Mathematical Society 133}, 6 (2005),
  1837--1841.

\bibitem{tian2007stochastic}
{\sc Tian, T., Burrage, K., Burrage, P.~M., and Carletti, M.}
\newblock Stochastic delay differential equations for genetic regulatory
  networks.
\newblock {\em Journal of Computational and Applied Mathematics 205}, 2 (2007),
  696--707.

\bibitem{renesse2010existence}
{\sc von Renesse, M.-K., and Scheutzow, M.}
\newblock Existence and uniqueness of solutions of stochastic functional
  differential equations.

\bibitem{WM08}
{\sc Wu, F., and Mao, X.}
\newblock {Numerical Solutions of Neutral Stochastic Functional Differential
  Equations}.
\newblock {\em SIAM Journal on Numerical Analysis 46}, 4 (2008), 1821--1841.

\bibitem{xie2020ergodicity}
{\sc Xie, L., and Zhang, X.}
\newblock Ergodicity of stochastic differential equations with jumps and
  singular coefficients.
\newblock In {\em Annales de l'Institut Henri Poincar{\'e}, Probabilit{\'e}s et
  Statistiques\/} (2020), vol.~56, Institut Henri Poincar{\'e}, pp.~175--229.

\bibitem{zhang2018singular}
{\sc Zhang, X., and Zhao, G.}
\newblock {Singular Brownian diffusion processes}.
\newblock {\em Communications in Mathematics and Statistics 6}, 4 (2018),
  533--581.

\end{thebibliography}
}

\end{document}